\documentclass[11pt]{amsart}

\usepackage{caption}
\DeclareCaptionLabelFormat{period}{Table \thetable.}
\usepackage[matrix,arrow,curve]{xy}
\usepackage{amsmath}
\usepackage{amssymb}
\usepackage{amsthm}
\usepackage{stmaryrd}
\usepackage[matrix,arrow,curve]{xy}
\usepackage{bbm}
\usepackage{float}
\usepackage{graphicx}
\usepackage{tikz}
\usepackage{tikz-cd}
\usetikzlibrary{calc}
\usetikzlibrary{arrows.meta}

\usepackage[colorlinks=true]{hyperref}

\renewcommand{\L}{\mathcal{L}}

\renewcommand{\O}{\mathcal{O}}
\newcommand{\F}{\mathcal{F}}

\renewcommand{\P}{\mathbb{P}}
\newcommand{\Z}{\mathbb{Z}}

\newcommand{\C}{\mathbb{C}}

\renewcommand{\dim}{\operatorname{dim}}

\newcommand{\Ext}{\operatorname{Ext}}

\newcommand{\rk}{\operatorname{rk}}

\renewcommand{\H}{\operatorname{H}}

\newtheorem{Lem}{Lemma}[section]
\newtheorem{Prop}[Lem]{Proposition}
\newtheorem{Cor}[Lem]{Corollary}
\newtheorem{Thm}[Lem]{Theorem}
\newtheorem*{Rem}{Remarks}

\newcommand{\Pk}{\operatorname{P^k}(\L_a)}

\usepackage{fancyhdr}
\usepackage{etoolbox}
\patchcmd{\section}{\scshape}{\bfseries}{}{}
\makeatletter
\renewcommand{\@secnumfont}{\bfseries}
\makeatother

\usepackage{etoolbox}
\makeatletter
\patchcmd{\@settitle}{\uppercasenonmath\@title}{}{}{}
\patchcmd{\@setauthors}{\MakeUppercase}{}{}{}

\makeatletter
\newcommand\xleftrightarrow[2][]{%
  \ext@arrow 9999{\longleftrightarrowfill@}{#1}{#2}}
\newcommand\longleftrightarrowfill@{%
  \arrowfill@\leftarrow\relbar\rightarrow}
\makeatother

\title{Ulrich Bundles on Elliptic Curves and Theta Functions}
\author{Alexander Pavlov}
\address{University of Wisconsin, 480 Lincoln Dr., Madison, WI, USA}
\email{pavlov@math.wisc.edu}
\thanks{
This material is based upon work supported by the National Science Foundation under grant No.0932078 000 while the second author was in residence at the Mathematical Sciences Research Institute in Berkeley, California, during the 2015/2106 year.}
\date{}



\begin{document}
\begin{abstract}
Let $E$ be a smooth elliptic curve over $\C$. For $E$ embedded into $\P^2$ as Hesse cubic and $V$ an Ulrich bundle on $E$ we derive a explicit presentation of $V$ using Moore matrices and theta functions.
\end{abstract}
\maketitle

\section{Introduction}

Let $(E,o)$ be a smooth elliptic curve over $\C$. Elliptic curve $E$ is a quotient $E \cong \C/ \Lambda$, where $\Lambda = \Z \oplus \Z \tau$. We denote $z$ the local coordinate on $E$. Theta functions $\theta_0(z), \theta_1(z)$ and $\theta_2(z)$ are global sections of $\O(1)=\O(3o)$, in fact they form a basis of $H^0(E, \O(3o))$. Here we ignore dependence of theta functions on parameter $\tau$ and consider $\tau$ to be fixed. We use these sections to embed curve $E$ into projective plane. It is well known (see, for example, \cite{Dolg}) that theta functions satisfy the following identity
$$
\theta_0^3(z)+\theta_1^3(z)+\theta_2^3(z)-3 \psi \theta_0(z) \theta_1(z) \theta_2(z) = 0,
$$
where $\psi \in \C$ and $\psi^3 \neq -1$ is a parameter depending on $\tau$. Thus, theta functions embed $E$ as a Hesse cubic $w$ in $\P^2$.

In this paper we study Ulrich bundles on the Hesse cubic $E \in \P^2$ that is vector bundles $V$ on $E$ of rank $r$ with presentation
$$
0 \rightarrow \O(-1)_{\P^2}^{3r} \xrightarrow {\ A\ } \O_{\P^2}^{3r} \rightarrow V \rightarrow 0.
$$

Rank one Ulrich bundles correspond to determinantal presentations of the Hesse cubic $x_0^3+x_1^3+x_2^3-3\psi x_0 x_x x_2=0$ and are given by Moore matrices
$$
M_{a,x}=\begin{pmatrix}
a_0x_0 & a_2x_2 & a_1x_1 \\
a_2x_1 & a_1x_0 & a_0x_2 \\
a_1x_2 & a_0x_1 & a_2x_0
\end{pmatrix},
$$
where $a=[a_0:a_1:a_2] \in E\setminus E[3]$ and $E[3]$ is the group of points of order $3$ in $E$. We interpret this observation in terms of theta functions then use automorphy factors to relate Ulrich bundles and theta functions (and derivatives of theta functions). In the last section we show that derivatives of theta functions can be eliminated form the formulas and get a presentation for indecomposable Ulrich bundles of any rank that does not explicitly use theta functions.

An indecomposable Ulrich bundle $\F_a^{k+1}$ of rank $k+1$ has the following presentation
\begin{Thm}
The Ulrich bundle $\F_a^{k+1}$ of rank $k+1$ over elliptic curve $E$ is the cokernel of the matrix $A_{k+1}$
$$
\xymatrix{
0 \ar[r] & \O_{\P^2}^{3(k+1)} (-1) \ar[r]^{A_{k+1}} & \O_{\P^2}^{3(k+1)} \ar[r] & \F_a^{k+1} \ar[r] & 0,
}
$$
where
$$
A_{k+1} = \begin{pmatrix} 
M_{a,x} &  {k \choose 1} M_{2a,x} & {k \choose 2} M_{4a,x} & \ldots & {k \choose {k-1}} M_{(-a)^{k-1}a,x} & M_{(-2)^ka,x}\\
0 & M_{a,x} & { {k-1} \choose 1} M_{-2a,x} & \ldots & {{k-1} \choose {k-2}} M_{(-2)^{k-2}a,x} & M_{(-2)^{k-1}a,x} \\
0 & 0 & M_{a,x} & \vdots & {{k-2} \choose {k-3}} M_{-2a,x} & M_{(-a)^{k-2}a,x} \\
\vdots & \vdots & \vdots & \ddots & \vdots & \vdots \\
0 & 0 & 0 & \ldots & M_{a,x} & M_{-2a,x} \\ 
0 & 0 & 0 & \ldots & 0 & M_{a,x}
\end{pmatrix}.
$$
\end{Thm}

All vector bundles in this paper are holomorphic. 

\section{Preliminaries}

In this section collect preliminaries used in the paper and give references for more detailed papers.

\subsection{Atiyah's Bundles over an Elliptic Curve}

On an elliptic curve $E$ there is a special discrete family of vector bundles, usually denoted by $F_r$ and parameterized by positive integer number $r \in \Z_{>0}$\,. By definition $F_1=\O$, and the vector bundle $F_r$ is defined inductively through the unique non-trivial extension:
$$
0 \to \O_E \to F_r \to F_{r-1} \to 0\,.
$$
It is easy to see that $\deg(F_r)=0$ and $\rk(F_r)=r$ for $r \in \Z_{>0}$\,. We call $F_r$, $r \in \Z_{>0}$, Atiyah bundles. For details, see \cite{Atiyah57}.

\subsection{Ulrich Bundles}

In this papers we only deal with Ulrich bundles on an elliptic curve embedded into projective plane $\P^2$ as Hesse Cubic. Thus we only define Ulrich bundle for hypersurfaces, for more general varieties see \cite{ES} and \cite{Beauville}. If $H$ is a hypersurface of degree $d$ in projective space $\P^n$. A vector bundle $V$ of rank $r$ on $H$ is called Ulrich if there is matrix $A$ of linear forms such that
$$
0 \rightarrow \O(-1)_{\P^n}^{rd} \xrightarrow {\ A\ } \O_{\P^n}^{rd} \rightarrow V \rightarrow 0.
$$
We call $A$ the matrix of the Ulrich Bundle $E$. Ulrich bundles can be characterized in terms of cohomology groups.

\begin{Prop} A vector bundle $V$ is Ulrich if and only if the cohomology $H^\bullet (H, V(-p))$ vanishes for $1\leq p\leq n-1.$
\end{Prop}

More details on Ulrich bundles can found in \cite{ES} and \cite{Beauville}.

\subsection{Moore Matrices and Theta Functions}

Let $a \neq 0$ be a point on the elliptic curve, and let $\L_a = t_{-a}^* \O(1)$. Line bundle $\L_a$ is a rank $1$ Ulrich bundle and any Ulrich bundle of rank $1$ can obtained in this way. The basis of sections of the line bundle $\L_a$ is the translated theta functions $\theta_i(z+a)$, where $i=0,1,2$. The pair of line bundles $\O(1)$ and $\L_a$ give an embedding of $E$ into $\P^2 \times \P^2$. If we use coordinates $x_i$ on the first copy of $\P^2$ and denote coordinates $y_i$ on the second copy of $\P^2$ then equations of $E$ in $\P^2 \times \P^2$ can we written using Moore matrix 
$$
M_{a,x}=\begin{pmatrix}
a_0x_0 & a_2x_2 & a_1x_1 \\
a_2x_1 & a_1x_0 & a_0x_2 \\
a_1x_2 & a_0x_1 & a_2x_0
\end{pmatrix}
$$
as 
$$
M_{a,x}\begin{pmatrix} y_0 \\ y_1 \\ y_2 \end{pmatrix} = 0.
$$
For details on Moore matrices see \cite{BP2015}. We get three equations of degree $(1,1)$ in $\P^2 \times \P^2$. Note that on the elliptic curve $a_i=\theta_i(a)$, $x_i=\theta_i(z)$ and $y_i=\theta_i(z+a)$ for $i=0,1,2$ then Moore matrix $M$ can be written as
$$
M_{a,x} = \begin{pmatrix}
\theta_0(a)\theta_0(z) & \theta_2(a)\theta_2(z) & \theta_1(a)\theta_1(z) \\
\theta_2(a)\theta_1(z) & \theta_1(a)\theta_0(z) & \theta_0(a)\theta_2(z) \\
\theta_1(a)\theta_2(z) & \theta_0(a)\theta_1(z) & \theta_2(a)\theta_0(z)
\end{pmatrix}.
$$
and we obtain three identities for theta functions:
\begin{gather*}
\theta_0(a)\theta_0(z)\theta_0(z+a)+\theta_2(a)\theta_2(z)\theta_1(z+a)+\theta_1(a)\theta_1(z)\theta_2(z+a)=0, \\
\theta_2(a)\theta_1(z)\theta_0(z+a)+\theta_1(a)\theta_0(z)\theta_1(z+a)+\theta_0(a)\theta_2(z)\theta_2(z+a)=0, \\
\theta_1(a)\theta_2(z)\theta_0(z+a)+\theta_0(a)\theta_1(z)\theta_1(z+a)+\theta_2(a)\theta_0(z)\theta_2(z+a)=0.
\end{gather*}
These identities are well known to specialists, cf. (\cite{O2002}, Appendix A).

\subsection{Sheaves of Principal Parts} 

Let $X$ be a smooth variety, we consider product $X \times X$ with two projections 

$$
\xymatrix{
& X \times X \ar[dl]_{\pi_1} \ar[dr]^{\pi_2}\\
X  && X}
$$

Let $\Delta: X \to X \times X$ be the diagonal embedding and $\O_\Delta = \Delta_*(\O_X)$, we set $I$ to be the sheaf of ideas of $\O_\Delta$

$$
0 \to I \to \O_{X \times X} \to \O_\Delta \to 0. 
$$

Let $V$ be a vector bundle on $X$ then the sheaf of principal parts (also known as jet bundles in differential geometry) $\operatorname{P}^k(V)$ of order $k \geq 0$ of vector bundle $V$ is defined by 
$$
\operatorname{P}^k(V) = (\pi_1)_*(\O_{X \times X} /I^{k+1} \otimes_{X \times X} \pi_2^*(V)).
$$

In particular, the sheaf of principal parts of order zero is $P^0(V)=V$.

For sheaves of principal parts there is short exact sequence
$$
0\to V \otimes_X \operatorname{Sym}^{k}(\Omega_X^1) \to \operatorname{P}^k(V) \to \operatorname{P}^{k-1}(V) \to 0.
$$

Details and proofs can be found in \cite{Perkinson}, appendix A.

\section{Automorphy factors of Ulrich bundles}

Isomorphism classes of vector bundles on $E$ of rank $r \geq 1$ are in bijection with first cohomology group of $\Lambda$ with coefficients in $GL_r(\O)$, see \cite{Iena2010}, so there is a bijection of sets:

$$
\left\{\begin{gathered}
\text{Isomorphism classes}\\
\text{of vector bundles of rank } r
\end{gathered}\right\} \xleftrightarrow{1:1} H^1(\Lambda, GL_r(\O)).
$$

A cohomology class in $\H^1(\Lambda, GL_r(\O))$ corresponding to a vector bundle $\F$ called the automorphy factor of the vector bundle $\F$.

Most commonly automorphy factors are used to describe line bundles. For a line bundle $\L$ the automorphy factor is a non-vanishing function $e(\lambda, z) \in \mathcal{O}^*(E)$ for every $\lambda \in \Lambda$, satisfying the cocycle condition
$$
e(\lambda+\mu,z)=e(\lambda, z+\mu)e(\mu, z),
$$
for $\mu, \lambda \in \Lambda$. In particular, one can use automorphy factors of line bundles to define theta functions. By definition, theta functions of the line bundle $\L$ satisfy condition
$$
e(\lambda, z)\theta(z) = \theta(z+\lambda).
$$

Our goal in this section is to compute the automorphy factors Ulrich bundles and describe analytically their sections in terms of theta functions. We will use these results to compute matrix factorizations of Ulrich bundles in the next section. We fix parameter $a \in E$, $a\neq o$ and $k \geq 1$.

\begin{Thm} Let $\L_a$ be an Ulrich bundle of rank one. Then any indecomposable Ulrich bundle $\F_a^r$ of rank $r=k+1$ can be obtained as a sheaf of principal parts 
$$
\F_a^r=\operatorname{P^k}(\L_a).
$$
\end{Thm}
\begin{proof}
We prove this statement by induction. We start with the $\rk(F_a)=2$. We have Atiyah's short exact sequence
$$
0 \to \L_a \to \operatorname{P^1}(\L_a) \to \L_a \to 0.
$$
From which we see that $\deg(\operatorname{P^1}(\L_a))=6$ and $\operatorname{P^1}(\L_a)$ is generated in degree 0 and, thus, it is Ulrich. We only need to show that it is indecomposable. For this we note that in the extension group $\Ext(\L_a, \L_a) \cong \C$ there is another generator obtained by twisting short exact sequence defining Atiyah's vector bundle $F_2$
$$
0 \to \O \to F_2 \to \O \to 0,
$$
by $\L_a$
$$
0 \to \L_a \to F_2 \otimes \L_a \to \L_a \to 0.
$$
Therefore, we have two generators in one-dimensional vector space $\Ext^1(\L_a, \L_a)$, they only can different by a multiplication on a non-zero constant $\rho$. Multiplication by $\rho$ induces a non-trivial automorphism of $\L_a$ and an isomorphism of two short exact sequences:
$$
\xymatrix{
0 \ar[r] & \L_a \ar[r]  \ar[d]_\rho^{\cong} &\operatorname{P^1}(\L_a) \ar[r] \ar[d]^{\cong} & \L_a \ar[r] \ar@{=}[d] & 0\\
0 \ar[r] & \L_a \ar[r]  &F_2 \otimes \L_a \ar[r] & \L_a \ar[r] & 0.
}
$$
Therefore we have isomorphism of the  middle terms
$$
\operatorname{P^1}(\L_a) \cong F_2 \otimes \L_a.
$$
Because we know that $F_2$ is indecomposable we conclude that $\operatorname{P^1}(\L_a)$ is indecomposable. Thus, we obtained a family of non-isomorphic indecomposable Ulrich bundles of rank two parametrized by $a \in E$, $a \neq o$. This construction give us all indecomposable Ulrich bundles of rank two.

Generalization for higher ranks is straightforward by induction if we notice that indecomposability follows from comparing two short exact sequences: for higher sheaves of principal parts
$$
0 \to \L_a \to \operatorname{P^k}(\L_a) \to \operatorname{P^{k-1}}(\L_a)  \to 0,
$$
and short exact sequence defining Atiyah's bundles $F_{k+1}$ tensored by $\L_a$
$$
0 \to \L_a \to F_{k+1} \otimes \L_a \to F_{k} \otimes \L_a \to 0.
$$
They again give us two nontrivial generators in the one-dimensional vector space $\Ext^1(\operatorname{P^{k}}(\L_a) , \L_a)$, which is isomorphic to  $\Ext^1(F_{k+1}(\L_a) , \L_a)$ by the induction hypothesis.
\end{proof}

Now we can obtain formulas for automorphy factors of Ulrich bundles of rank $k$. The proof of the following corollary is a standard computation with principal parts, we include it for completeness of the exposition. 

\begin{Cor}
If $e_a(\lambda, z)$ is the automorphy factor of the Ulrich line bundle $\L_a$ then the automorphy factor of $\F_a^{k+1} =  \operatorname{P^k}(\L_a)$ is 
$$
f_a(\lambda,z) = 
\begin{pmatrix} 
e_a(\lambda,z) &  {k \choose 1} e'_a(\lambda, z) & {k \choose 2} e''_a(\lambda, z) & \ldots & {k \choose {k-1}} e^{(k-1)}_a(\lambda,z) & e^{(k)}_a(\lambda,z) \\
0 & e_a(\lambda,z) & { {k-1} \choose 1} e'_a(\lambda, z) & \ldots & {{k-1} \choose {k-2}} e^{(k-2)}_a(\lambda, z) & e_a^{(k-1)}(\lambda,z) \\
0 & 0& e_a(\lambda,z) & \vdots & {{k-2} \choose {k-3}} e^{(k-3)}_a(\lambda,z) & e^{(k-2)}_a(\lambda,z) \\
\vdots & \vdots & \vdots & \ddots & \vdots & \vdots \\
0 & 0 & 0 & \ldots & e_a(\lambda, z) & e'_a(\lambda ,z) \\ 
0 & 0 & 0 & \ldots & 0 & e_a(\lambda, z) 
\end{pmatrix},
$$
where we use short notation for derivatives $e^{(i)}_a(\lambda,z) = \dfrac{d^i }{dz^i} e_a(\lambda,z)$.
\end{Cor}

\begin{proof}
Let $\theta(z)$ be a section of $\L_a$, then $e_a(\lambda, z)\theta(z) = \theta(z+\lambda)$. On the product of elliptic curve with itself $E_w \times E_z$ we use coordinates $w$ and $z$ as indicated by subscript. Let $t=z-w$ then equation of the $k$-th infinitesimal neighbourhood of the diagonal is $t^{k+1}=0$. Taylor expending
$$
e_a(\lambda, w+t)\theta(w+t) = \theta(w+t+\lambda)
$$
into Taylor polynomial of order $k$ we get the automorphy factor $f_a(\lambda,z) $ for $\sum_{i=0}^{k} \dfrac{ \theta^{(i)}(w) t^i}{i!}$, which is a section of $\Pk$.
\end{proof}

Now it is easy to find a basis of global sections of $\F_a^{k+1}$. Analytically, section of $\F_a^{k+1}$ is a vector function 
$$
v(z)=\begin{pmatrix} v_0(z)  \\ v_1(z)  \\ \vdots \\ v_k(z) \end{pmatrix},
$$
such that
$$
f_a(\lambda, z) v(z) = v(z+\lambda).
$$

\begin{Cor}The vector bundle $\Pk$ has $3(k+1)$ global sections. A basis of global section can be chosen in the following form
$$
\begin{pmatrix} \theta_i(z+a)  \\ 0  \\ 0 \\ \vdots \\ 0 \end{pmatrix}, 
\begin{pmatrix} \theta'_i(z+a)  \\ \dfrac{1}{k} \theta_i(z+a) \\ 0 \\ \vdots \\ 0 \end{pmatrix}, 
\begin{pmatrix} \theta''_i(z+a)  \\ \dfrac{2}{k} \theta'_i(z+a)  \\ \dfrac{1}{k-1} \theta_i(z+a) \\ \vdots \\ 0 \end{pmatrix}, 
\hdots, 
\begin{pmatrix} \theta_i^{(k)}(z+a)  \\ \theta_i^{(k-1)}(z+a) \\ \theta_i^{(k-2)}(z+a) \\ \vdots \\ \theta_i(z+a) \end{pmatrix},
$$
where $i=0,1,2$.
\end{Cor}

\begin{proof}
It is easy to check by direct computation that all this analytic vector functions satisfy automorphy condition with automorphy factor $f_a(\lambda,z)$ obtained above. They are clearly linearly independent and there $3(k+1)$ of them: $k+1$ vectors and $i=0,1,2$.
\end{proof}

\textbf{Example: rank two case.}
For $k=1$ the automorphy factor of $\F_a^2$ is 
$$
f_a(\lambda,z) = \begin{pmatrix} e_a(\lambda, z) & \dfrac{d e_a(\lambda ,z)}{dz} \\ 0 & e_a(\lambda, z) \end{pmatrix}.
$$

The basis of global sections of $\F_a^2$ is given by vector functions
$$
v(z)=\begin{pmatrix} v_1(z)  \\ v_2(z) \end{pmatrix},
$$
such that
$$
f_a(\lambda, z) v(z) = v(z+\lambda).
$$

Dimension $\dim H^0(E, \F_a^2) = \deg(\F_a^2)=6$. Six linearly independent global sections of $\F_a^2$ are 
$$
\begin{pmatrix} \dfrac{d \theta_{i}(z)}{dz} \\ \theta_{i}(z) \end{pmatrix}, \qquad \begin{pmatrix} \theta_{i}(z) \\ 0 \end{pmatrix}, 
$$
for $i=0,1,2$.

\section{Martices of Ulrich bundle and derivatives of theta functions} 

The goal of this section is to find matrix factorizations of Ulrich bundles. As Ulrich bundles are principal sheaves bundles of an Ulrich line bundle it come with no surprise that matrix factorizations involve derivatives of the theta functions. First, we carry the computation for rank two Ulrich bundle, where some of the steps a much clearer. After this we generalize the formulas to any rank.

For a line Ulrich bundle $\L_a$ we have an epimorphism induces by the evaluation map on the zeroth cohomology
$$
H^0(E, \L_a) \otimes H^0(E, \O(1)) \to H^0(E, \L_a(1)) \to 0.
$$
Three dimensional kernel of this map is embedded into $H^0(E, \L_a) \otimes H^0(E, \O(1))$ by the Moore matrix $M$.

We repeat this construction for the bundle $\F_a^2$ of rank two. The evaluation map again induce epimorpism on zeroth cohomology
$$
H^0(E, \F_a^2) \otimes H^0(E, \O(1)) \to H^0(E, \F_a^2(1)) \to 0.
$$
By computing dimensions
$$
\dim H^0(E, \F_a^2) \otimes H^0(E, \O(1)) = \dim H^0(E, \F_a) \dim H^0(E, \O(1)) = 18,
$$
and 
$$
\dim H^0(E, \F_a^2(1)) = \deg \F_a^2(1) = 12.
$$
we see that this map has a six dimensional kernel, or, equivalently, there are six relations between sections of $\F_a^2$ and sections of $\O(1)$. We get this new six relations by differentiating relations given by Moore matrix
$$
M_{a,x} \begin{pmatrix} \theta_0(z+a) \\ \theta_1(z+a) \\ \theta_2(z+a) \end{pmatrix} =0,
$$
by parameter $a$. We get
$$
M_{a,x} \begin{pmatrix} \theta_0'(z+a) \\ \theta_1'(z+a) \\ \theta_2'(z+a) \end{pmatrix} + M'_{a,x} \begin{pmatrix} \theta_0(z+a) \\ \theta_1(z+a) \\ \theta_2(z+a) \end{pmatrix} =0,
$$
where 
$$
M'_{a,x} = \frac{\partial}{\partial a} M_{a,x} = \begin{pmatrix}
\theta'_0(a)\theta_0(z) & \theta'_2(a)\theta_2(z) & \theta'_1(a)\theta_1(z) \\
\theta'_2(a)\theta_1(z) & \theta'_1(a)\theta_0(z) & \theta'_0(a)\theta_2(z) \\
\theta'_1(a)\theta_2(z) & \theta'_0(a)\theta_1(z) & \theta'_2(a)\theta_0(z)
\end{pmatrix}.
$$
We can also write this identities as
\begin{multline*}
\theta_0(a)\theta_0(z)\theta'_0(z+a)+\theta_2(a)\theta_2(z)\theta'_1(z+a)+\theta_1(a)\theta_1(z)\theta'_2(z+a)+\\
\shoveright{\theta'_0(a)\theta_0(z)\theta_0(z+a)+\theta'_2(a)\theta_2(z)\theta_1(z+a)+\theta'_1(a)\theta_1(z)\theta_2(z+a) = 0,}\\
\shoveleft{\theta_2(a)\theta_1(z)\theta'_0(z+a)+\theta_1(a)\theta_0(z)\theta'_1(z+a)+\theta_0(a)\theta_2(z)\theta'_2(z+a)+}\\
\shoveright{\theta'_2(a)\theta_1(z)\theta_0(z+a)+\theta'_1(a)\theta_0(z)\theta_1(z+a)+\theta'_0(a)\theta_2(z)\theta_2(z+a) = 0,}\\
\shoveleft{\theta_1(a)\theta_2(z)\theta'_0(z+a)+\theta_0(a)\theta_1(z)\theta'_1(z+a)+\theta_2(a)\theta_0(z)\theta'_2(z+a)+}\\
\theta'_1(a)\theta_2(z)\theta_0(z+a)+\theta'_0(a)\theta_1(z)\theta_1(z+a)+\theta'_2(a)\theta_0(z)\theta_2(z+a) = 0.
\end{multline*}

\begin{Prop}
The kernel of the map 
$$
H^0(E, \F_a^2) \otimes H^0(E, \O(1)) \to H^0(E, \F_a^2(1)) \to 0.
$$
is generated by the following $6$ linearly independent linear relations
$$
\begin{pmatrix}
M_{a,x} & M'_{a,x} \\
0 & M_{a,x}
\end{pmatrix} \otimes \operatorname{id}_2 
\begin{pmatrix}
\left(
\begin{gathered}
\theta'_0(z+a)\\
\theta_0(z+a)
\end{gathered}
\right)\\
\left(
\begin{gathered}
\theta'_1(z+a)\\
\theta_1(z+a)
\end{gathered}
\right)\\
\left(
\begin{gathered}
\theta'_2(z+a)\\
\theta_2(z+a)
\end{gathered}
\right)\\
\left(
\begin{gathered}
\theta_0(z+a)\\
0
\end{gathered}
\right)\\
\left(
\begin{gathered}
\theta_1(z+a)\\
0
\end{gathered}
\right)\\
\left(
\begin{gathered}
\theta_2(z+a)\\
0
\end{gathered}
\right)\\
\end{pmatrix} = 0.
$$

where 
$$
\begin{pmatrix}
M_{a,x} & M'_{a,x} \\
0 & M_{a,x}
\end{pmatrix}=\begin{pmatrix}
\theta_0(a)\theta_0(z) & \theta_2(a)\theta_2(z) & \theta_1(a)\theta_1(z) & \theta'_0(a)\theta_0(z) & \theta'_2(a)\theta_2(z) & \theta'_1(a)\theta_1(z) \\
\theta_2(a)\theta_1(z) & \theta_1(a)\theta_0(z) & \theta_0(a)\theta_2(z) & \theta'_2(a)\theta_1(z) & \theta'_1(a)\theta_0(z) & \theta'_0(a)\theta_2(z) \\
\theta_1(a)\theta_2(z) & \theta_0(a)\theta_1(z) & \theta_2(a)\theta_0(z) & \theta'_1(a)\theta_2(z) & \theta'_0(a)\theta_1(z) & \theta'_2(a)\theta_0(z) \\
0 & 0 & 0 & \theta_0(a)\theta_0(z) & \theta_2(a)\theta_2(z) & \theta_1(a)\theta_1(z) \\
0 & 0 & 0 & \theta_2(a)\theta_1(z) & \theta_1(a)\theta_0(z) & \theta_0(a)\theta_2(z) \\
0 & 0 & 0 & \theta_1(a)\theta_2(z) & \theta_0(a)\theta_1(z) & \theta_2(a)\theta_0(z)
\end{pmatrix}.
$$
\end{Prop}

\begin{proof}
We need to check $12$ identities: $6$ for first component of sections of $\F_a^2$ and $6$ for second component of sections of $\F_a^2$. But all identities that we need to verify are exactly
$$
M_{a,x} \begin{pmatrix} \theta_0(z+a) \\ \theta_1(z+a) \\ \theta_2(z+a) \end{pmatrix} =0,
$$
and
$$
M_{a,x} \begin{pmatrix} \theta_0'(z+a) \\ \theta_1'(z+a) \\ \theta_2'(z+a) \end{pmatrix} + M' _{a,x}\begin{pmatrix} \theta_0(z+a) \\ \theta_1(z+a) \\ \theta_2(z+a) \end{pmatrix} =0.
$$
Linear independence is clear.
\end{proof}

We immediately get
\begin{Cor}
Ulrich vector bundle $\F_a^2$ is the cokernel of the $6 \times 6$ block matrix
$$
A = \begin{pmatrix}
M_{a,x} & M'_{a,x} \\
0 & M_{a,x}
\end{pmatrix},
$$
where $M_{a,x}$ is the Moore matrix, and $M'_{a,x}=\frac{\partial}{\partial a} M_{a,x}$.
\end{Cor}

Now we are in position to obtain matrix factorization of $\F_a^2$. Let us denote second matrix in the matrix factorization of $\L_a$ by $L_{a,x}$. So, ordered pair $(M_{a,x}, L_{a,x})$ is the matrix factorization of $\L_a$. Recall that 
$$
L_{a,x}=\frac{1}{a_0 a_1 a_2} \begin{pmatrix}
a_1a_2x_0^2 -a_0^2x_1x_2 & a_0a_1x_1^2-a_2^2x_0x_2 & a_0a_2x_2^2-a_1^2x_0x_1 \\
a_0a_1x_2^2-a_2^2x_0x_1 & a_0a_2x_0^2-a_1^2x_1x_2 & a_1a_2x_1^2 -a_0^2x_0x_2 \\
a_0a_2x_1^2-a_1^2x_0x_2 & a_1a_2x_2^2 -a_0^2x_0x_1 & a_0a_1x_0^2-a_2^2x_1x_2
\end{pmatrix}.
$$
\begin{Cor}
The matrix factorization of $\F_a^2$ is 
$$
A_2 = \begin{pmatrix}
M_{a,x} & M'_{a,x} \\
0 & M_{a,x}
\end{pmatrix}, \qquad B_2 = \begin{pmatrix} L _{a,x}& L'_{a,x} \\ 0 & L_{a,x} \end{pmatrix},
$$
where $L$ is as above but rewritten in terms of theta functions 
and $L' = \frac{\partial}{\partial a} L$.
\end{Cor}

\begin{proof}
We only need to compute upper right block of the matrix $B$. This is done differentiating identity $M_{a,x} L_{a,x}=w$ by parameter $a$
$$
M' _{a,x}L_{a,x} +  M_{a,x} L'_{a,x} = \frac{\partial}{\partial a} w = 0.
$$
\end{proof}

Now we generalize this result to an Ulrich bundle of any rank. 

\begin{Thm}
Let $\F_a^{k+1} = \Pk$ be an Ulrich bundle of rank $r=k+1$, then matrix factorization of $\F_a^{k+1}$ is
$$
A _{k+1}= \begin{pmatrix} 
M_{a,x} &  {k \choose 1} M_{a,x}' & {k \choose 2} M_{a,x}'' & \ldots & {k \choose {k-1}} M_{a,x}^{(k-1)} & M_{a,x}^{(k)} \\
0 & M_{a,x} & { {k-1} \choose 1} M_{a,x}' & \ldots & {{k-1} \choose {k-2}} M_{a,x}^{(k-2)} & M_{a,x}^{(k-1)} \\
0 & 0 & M_{a,x} & \vdots & {{k-2} \choose {k-3}} M_{a,x}^{(k-3)} & M_{a,x}^{(k-2)} \\
\vdots & \vdots & \vdots & \ddots & \vdots & \vdots \\
0 & 0 & 0 & \ldots & M_{a,x} & M_{a,x}' \\ 
0 & 0 & 0 & \ldots & 0 & M_{a,x}
\end{pmatrix}
$$
and
$$
B_{k+1} = \begin{pmatrix} 
L_{a,x} &  {k \choose 1} L_{a,x}' & {k \choose 2} L_{a,x}'' & \ldots & {k \choose {k-1}} L_{a,x}^{(k-1)} & L_{a,x}^{(k)} \\
0 & L_{a,x} & { {k-1} \choose 1} L_{a,x}' & \ldots & {{k-1} \choose {k-2}} L_{a,x}^{(k-2)} & L_{a,x}^{(k-1)} \\
0 & 0 & L_{a,x} & \vdots & {{k-2} \choose {k-3}} L_{a,x}^{(k-3)} & L_{a,x}^{(k-2)} \\
\vdots & \vdots & \vdots & \ddots & \vdots & \vdots \\
0 & 0 & 0 & \ldots & L_{a,x} & L_{a,x}' \\ 
0 & 0 & 0 & \ldots & 0 & L_{a,x}
\end{pmatrix}.
$$
In particular, Ulrich bundle $\F_a^{k+1}$ is a cokernel of $A_{k+1}$
$$
\xymatrix{
0 \ar[r] & \O_{\P^2}^{3(k+1)} (-1) \ar[r]^{A_{k+1}} & \O_{\P^2}^{3(k+1)} \ar[r] & \F_a^{k+1} \ar[r] & 0.
}
$$
\end{Thm}

\begin{proof}
As before the evaluation map induce an epimorpism on zeroth cohomology
$$
H^0(E, \F_a^{k+1}) \otimes H^0(E, \O(1)) \to H^0(E, \F_a^{k+1}(1)) \to 0.
$$
Dimension of the kernel of this map is $3(k+1)$, so we need to find $3(k+1)$ linear relations among global sections of $\F_a^{k+1}$.

Let $\theta$ denote vector of theta functions
$$
\theta = \begin{pmatrix}  \theta_0(z+a) \\ \theta_1(z+a) \\ \theta_2(z+a) \end{pmatrix}.
$$
Identity for theta functions could now written as
$$
M_{a,x} \theta = 0. 
$$
First, differentiating this identity $k$ times we obtain $k+1$ relations for global sections of $\F_a^{k+1}$ given by the matrix $A_{k+1}$ that is
$$
\begin{pmatrix} 
M_{a,x} &  {k \choose 1} M_{a,x}' & {k \choose 2} M_{a,x}'' & \ldots & {k \choose {k-1}} M_{a,x}^{(k-1)} & M_{a,x}^{(k)} \\
0 & M_{a,x} & { {k-1} \choose 1} M_{a,x}' & \ldots & {{k-1} \choose {k-2}} M_{a,x}^{(k-2)} & M_{a,x}^{(k-1)} \\
0 & 0 & M_{a,x} & \vdots & {{k-2} \choose {k-3}} M_{a,x}^{(k-3)} & M_{a,x}^{(k-2)} \\
\vdots & \vdots & \vdots & \ddots & \vdots & \vdots \\
0 & 0 & 0 & \ldots & M_{a,x} & M_{a,x}' \\ 
0 & 0 & 0 & \ldots & 0 & M_{a,x}
\end{pmatrix}
\begin{pmatrix}
\theta^{k} \\ \theta^{(k-1)} \\ \theta^{(k-2)} \\ \vdots \\ \theta' \\ \theta 
\end{pmatrix}
=0.
$$
Second, by the same computation as for $k=1$ case the basis of global sections
$$
\begin{pmatrix} \theta_i(z+a)  \\ 0  \\ 0 \\ \vdots \\ 0 \end{pmatrix}, 
\begin{pmatrix} \theta'_i(z+a)  \\ \dfrac{1}{k} \theta_i(z+a) \\ 0 \\ \vdots \\ 0 \end{pmatrix}, 
\begin{pmatrix} \theta''_i(z+a)  \\ \dfrac{2}{k} \theta'_i(z+a)  \\ \dfrac{1}{k-1} \theta_i(z+a) \\ \vdots \\ 0 \end{pmatrix}, 
\hdots, 
\begin{pmatrix} \theta_i^{(k)}(z+a)  \\ \theta_i^{(k-1)}(z+a) \\ \theta_i^{(k-2)}(z+a) \\ \vdots \\ \theta_i(z+a) \end{pmatrix},
$$
(where $i=0,1,2$) is in the kernel of the matrix $A_{k+1} \otimes \operatorname{id}_{k+1}$. 

Clearly this $3(k+1)$ linear relations represented by rows of the matrix $A_{k+1}$ are linearly independent.
\end{proof}

\section{Eliminating derivatives of theta functions}

In this section we express derivatives $\theta_i'(a)$ in terms of theta functions $\theta_i(a)$ themselves. This allow us to write a presentation  of the vector bundle $\F_a^{k+1}$ in terms of parameters $a_i=\theta_i(a)$ and completely eliminate theta functions and their derivatives from the formulas. 

\begin{Thm}The following relation between derivatives of theta functions and theta functions at the double argument is true
$$
\begin{pmatrix}
\theta'_0(a)\\
\theta'_2(a)\\
\theta'_1(a)
\end{pmatrix}
\doteq
s(a) \begin{pmatrix}
\theta_0(a)\\
\theta_2(a)\\
\theta_1(a)
\end{pmatrix}
+
\begin{pmatrix}
\theta_0(-2a)\\
\theta_2(-2a)\\
\theta_1(-2a)
\end{pmatrix},
$$
where $s(a)$ is a undermined parameter depending on $a$ and $\doteq$ means that two sides are equal up to a nonzero scalar which does not depend on $a$.
\end{Thm}

\begin{proof}
Our identities for derivatives of theta functions when $z=-a$ give the following system of linear equations
\begin{gather*}
\begin{pmatrix}
\theta_0 \theta_0(-a)  & \theta_1 \theta_2(-a) & \theta_2 \theta_1(-a)\\
\theta_2 \theta_2(-a) & \theta_0 \theta_1(-a) & \theta_1 \theta_0(-a)\\
\theta_1 \theta_1(-a) & \theta_2 \theta_0(-a) & \theta_0 \theta_1(-a)
\end{pmatrix}
\begin{pmatrix}
\theta'_0(a)\\
\theta'_2(a)\\
\theta'_1(a)
\end{pmatrix} = \\
-\begin{pmatrix}
\theta_0(a) \theta_0(-a) \theta'_0 + \theta_2(a)\theta_2(-a)\theta'_1 + \theta_1(a)\theta_1(-a)\theta'_2\\
\theta_2(a) \theta_1(-a) \theta'_0 + \theta_1(a)\theta_0(-a)\theta'_1 + \theta_0(a)\theta_2(-a)\theta'_2\\
\theta_1(a) \theta_2(-a) \theta'_0 + \theta_0(a)\theta_1(-a)\theta'_1 + \theta_2(a)\theta_0(-a)\theta'_2
\end{pmatrix},
\end{gather*}

where we use conventions $\theta_i=\theta_i(0)$ and $\theta_i'=\theta_i'(0)$. Values of theta functions at the point $-a$ are related to the values at the point $a$ as 
 \begin{gather*}
\theta_0(-a) = - \theta_0(a), \\
\theta_1(-a) = - \theta_2(a), \\
\theta_2(-a) = - \theta_1(a).
\end{gather*}
These identities immediately follow from a presentation of theta functions as a series, cf \cite{O2002}. 

Theta functions map origin of $\C$ to the origin of $E$:
$$
[\theta_0:\theta_1:\theta_2] = [0:1:-1].
$$
Therefore we can eliminate theta constants $\theta_0$ and $\theta_2$ from the left hand side of the system
$$
\theta_1 \begin{pmatrix}
0 & \theta_1(a) & -\theta_2(a)\\
-\theta_1(a) & 0 & \theta_0(a)\\
\theta_2(a) & -\theta_0(a) & 0
\end{pmatrix}
\begin{pmatrix}
\theta'_0(a)\\
\theta'_2(a)\\
\theta'_1(a)
\end{pmatrix}
=
-\begin{pmatrix}
\theta_0^2(a)\theta'_0 + \theta_1(a)\theta_2(a)\theta'_1 + \theta_1(a)\theta_2(a)\theta'_2\\
\theta_2^2(a)\theta'_0 + \theta_0(a)\theta_1(a)\theta'_1 + \theta_0(a)\theta_1(a)\theta'_2\\
\theta_1^2(a)\theta'_0 + \theta_0(a)\theta_2(a)\theta'_1 + \theta_0(a)\theta_2(a)\theta'_2
\end{pmatrix}.
$$

We drop the theta constant $\theta_1$ from the system by changing $=$ sign to $\doteq$, because we only need a solution up to a multiplication on a non-zero scalar
$$
\begin{pmatrix}
0 & \theta_1(a) & -\theta_2(a)\\
-\theta_1(a) & 0 & \theta_0(a)\\
\theta_2(a) & -\theta_0(a) & 0
\end{pmatrix}
\begin{pmatrix}
\theta'_0(a)\\
\theta'_2(a)\\
\theta'_1(a)
\end{pmatrix}
=
-\begin{pmatrix}
\theta_0^2(a)\theta'_0 + \theta_1(a)\theta_2(a)\theta'_1 + \theta_1(a)\theta_2(a)\theta'_2\\
\theta_2^2(a)\theta'_0 + \theta_0(a)\theta_1(a)\theta'_1 + \theta_0(a)\theta_1(a)\theta'_2\\
\theta_1^2(a)\theta'_0 + \theta_0(a)\theta_2(a)\theta'_1 + \theta_0(a)\theta_2(a)\theta'_2
\end{pmatrix}.
$$ 

Differentiating identity 
$$
\theta_0^3(z)+\theta_1^3(z)+\theta_2^3(z)-3 \psi \theta_0(z) \theta_1(z) \theta_2(z) = 0
$$
at $z=0$ we get
$$
\psi \theta_0'+\theta_1'+\theta_2'=0.
$$

We use this to rewrite right hand side of the system
\begin{gather*}
-\begin{pmatrix}
\theta_0^2(a)\theta'_0 + \theta_1(a)\theta_2(a)\theta'_1 + \theta_1(a)\theta_2(a)\theta'_2\\
\theta_2^2(a)\theta'_0 + \theta_0(a)\theta_1(a)\theta'_1 + \theta_0(a)\theta_1(a)\theta'_2\\
\theta_1^2(a)\theta'_0 + \theta_0(a)\theta_2(a)\theta'_1 + \theta_0(a)\theta_2(a)\theta'_2
\end{pmatrix}=-\theta_0'
\begin{pmatrix}
\theta_0^2(a) - \psi \theta_1(a)\theta_2(a)\\
\theta_2^2(a) - \psi \theta_0(a)\theta_1(a)\\
\theta_1^2(a) - \psi \theta_0(a)\theta_2(a)
\end{pmatrix}
\doteq \\
\begin{pmatrix}
\psi \theta_1(a)\theta_2(a) - \theta_0^2(a)\\
\psi \theta_0(a)\theta_1(a) - \theta_2^2(a) \\
\psi \theta_0(a)\theta_2(a) - \theta_1^2(a)
\end{pmatrix}=
\begin{pmatrix}
\dfrac{\theta_1(a)^3+\theta_2(a)^3-2\theta_0(a)^3}{\theta_0(a)}\\
\dfrac{\theta_0(a)^3+\theta_1(a)^3-2\theta_2(a)^3}{\theta_2(a)}\\
\dfrac{\theta_0(a)^3+\theta_2(a)^3-2\theta_1(a)^3}{\theta_1(a)}
\end{pmatrix}.
\end{gather*}
Combining this together we get the following system of equations for derivatives of theta functions
$$
\begin{pmatrix}
0 & \theta_1(a) & -\theta_2(a)\\
-\theta_1(a) & 0 & \theta_0(a)\\
\theta_2(a) & -\theta_0(a) & 0
\end{pmatrix}
\begin{pmatrix}
\theta'_0(a)\\
\theta'_2(a)\\
\theta'_1(a)
\end{pmatrix}
\doteq
\begin{pmatrix}
\dfrac{\theta_1(a)^3+\theta_2(a)^3-2\theta_0(a)^3}{\theta_0(a)}\\
\dfrac{\theta_0(a)^3+\theta_1(a)^3-2\theta_2(a)^3}{\theta_2(a)}\\
\dfrac{\theta_0(a)^3+\theta_2(a)^3-2\theta_1(a)^3}{\theta_1(a)}
\end{pmatrix}.
$$
One can easily check that vector
$$
\begin{pmatrix}
\theta_0(-2a)\\
\theta_2(-2a)\\
\theta_1(-2a)
\end{pmatrix}
\doteq
\begin{pmatrix}
\dfrac{\theta_2(a)^3-\theta_1(a)^3}{\theta_1(a)\theta_2(a)}\\
\dfrac{\theta_1(a)^3-\theta_0(a)^3}{\theta_0(a)\theta_1(a)}\\
\dfrac{\theta_0(a)^3-\theta_2(a)^3}{\theta_0(a)\theta_2(a)}
\end{pmatrix}
$$
is a particular solution of the non-homogeneous system and the vector
$$
s(a) \begin{pmatrix}
\theta_0(a)\\
\theta_2(a)\\
\theta_1(a)
\end{pmatrix}
$$
is the general solution of the homogeneous system. Here $s(a)$ is a parameter that depends on the point $a$. Thus the general solution of the system is
$$
\begin{pmatrix}
\theta'_0(a)\\
\theta'_2(a)\\
\theta'_1(a)
\end{pmatrix}
\doteq
s(a) \begin{pmatrix}
\theta_0(a)\\
\theta_2(a)\\
\theta_1(a)
\end{pmatrix}
+
\begin{pmatrix}
\theta_0(-2a)\\
\theta_2(-2a)\\
\theta_1(-2a)
\end{pmatrix},
$$
where $s(a)$ is a parameter.
\end{proof}

\begin{Rem} 
It is natural and expected that a formula for the derivatives of theta functions involve doubling of the argument. Because a line tangent to a smooth elliptic curve at a point $a$ also intersect the elliptic curve at the point $-2a$.
\end{Rem}

In order to apply this theorem to matrix factorizations we reformulate it in terms of Moore matrices. 

\begin{Cor}
$M'_{a,x} \doteq s(a) M_{a,x}+M_{-2a,x}$.
\end{Cor}
 
In the solution we have two undetermined parameters: we found the solution only up to a non-zero constant and also we have parameter $s(a)$. But we don't need to know these parameters to find matrix factorizations: we need a non-trivial extension, which particular extension we use is not important and thus common non-zero scalar does not play any role, parameter $s(a)$ can always be eliminated by row/column operations.

Let us illustrate this by the rank two case. We have
$$
A_2 = \begin{pmatrix}
M_{a,x} & M'_{a,x} \\
0 & M_{a,x}
\end{pmatrix} \sim
\begin{pmatrix}
M_{a,x} & M'_{a,x} - s(a) M_{a,x} \\
0 & M_{a,x}
\end{pmatrix} =
\begin{pmatrix}
M_{a,x} & M_{-2a,x} \\
0 & M_{a,x}
\end{pmatrix}.
$$

Therefore, we get
\begin{Cor}
The Ulrich bundle $\F_a^2$ of rank $2$ over elliptic curve $E$ is the cokernel of the matrix $A$
$$
\xymatrix{
0 \ar[r] & \O_{\P^2}^6 (-1) \ar[r]^{A_2} & \O_{\P^2}^6 \ar[r] & \F_a^2 \ar[r] & 0,
}
$$
where
$$
A_2 = \begin{pmatrix}
M_{a,x} & M_{-2a,x} \\
0 & M_{a,x}
\end{pmatrix}.
$$
\end{Cor}

This result for rank two Ulrich bundles was proven in \cite{BP2015} by different methods.

For higher rank matrix factorizations we need to generalize the formula for derivatives.

\begin{Cor}
$$
M_{a,x}^{(k+1)} = \sum_{l=0}^{k-1} \sum_{i=0}^{k-l} s^{(i)}(a) M_{(-2)^l a,x}^{(k-l-i)}+M_{(-2)^{k+1}a,x}
$$
\end{Cor}
\begin{proof}
We apply formula $M'_{a,x} \doteq s(a) M_{a,x}+M_{-2a,x}$ inductively $k$ times 
\begin{gather*}
M_{a,x}^{k+1} \doteq \sum_{i=0}^{k} s^{(i)}(a) M_{a.x}^{(k-i)} + M_{-2a,x}^{(k)} \doteq \\
\sum_{i=0}^{k} s^{(i)}(a) M_{a,x}^{(k-i)} + \sum_{i=0}^{k-1} s^{(i)}(a) M_{a.x}^{(k-i-1)} + M_{4a,x}^{(k-1)} \doteq 
\hdots \\ 
\doteq \sum_{l=0}^{k-1} \sum_{i=0}^{k-l} s^{(i)}(a) M_{(-2)^l a,x}^{(k-l-i)}+M_{(-2)^{k+1}a,x}.
\end{gather*}
\end{proof}

Finally we get a presentation of the Ulrich bundle $\F_a^{k+1}$ of rank $k+1$.

\begin{Cor}
The Ulrich bundle $\F_a^{k+1}$ of rank $k+1$ over elliptic curve $E$ is the cokernel of the matrix $A_{k+1}$
$$
\xymatrix{
0 \ar[r] & \O_{\P^2}^{3(k+1)} (-1) \ar[r]^{A_{k+1}} & \O_{\P^2}^{3(k+1)} \ar[r] & \F_a^{k+1} \ar[r] & 0,
}
$$
where
$$
A_{k+1} = \begin{pmatrix} 
M_{a,x} &  {k \choose 1} M_{2a,x} & {k \choose 2} M_{4a,x} & \ldots & {k \choose {k-1}} M_{(-a)^{k-1}a,x} & M_{(-2)^ka,x}\\
0 & M_{a,x} & { {k-1} \choose 1} M_{-2a,x} & \ldots & {{k-1} \choose {k-2}} M_{(-2)^{k-2}a,x} & M_{(-2)^{k-1}a,x} \\
0 & 0 & M_{a,x} & \vdots & {{k-2} \choose {k-3}} M_{-2a,x} & M_{(-a)^{k-2}a,x} \\
\vdots & \vdots & \vdots & \ddots & \vdots & \vdots \\
0 & 0 & 0 & \ldots & M_{a,x} & M_{-2a,x} \\ 
0 & 0 & 0 & \ldots & 0 & M_{a,x}
\end{pmatrix}
$$
\end{Cor}
\begin{proof}
This follows by induction from previous formula and column row operations eliminating parameter $s(a)$ and its derivatives.
\end{proof}


\begin{thebibliography}{99}

\bibitem{Atiyah57} Atiyah M., \emph{Vector Bundles over an Elliptic Curve}, Proc. London Math. Soc. 7(3) (1957), 415-452.

\bibitem{Beauville} Beauville A. \emph{An Introduction to Ulrich Bundles}, preprint, \href{https://arxiv.org/abs/1610.02771}{arxiv.org/abs/1610.02771}

\bibitem{BP2015} Buchweitz R.-O., Pavlov A., \emph{Moore Matrices and Ulrich Bundles on an Elliptic Curve}, preprint, \href{https://arxiv.org/abs/1511.05502}{arxiv.org/abs/1511.05502}

\bibitem{Dolg} Dolgachev I. V., \emph{Lectures on modular forms}, \href{http://www.math.lsa.umich.edu/~idolga}{www.math.lsa.umich.edu/~idolga}.

\bibitem{ES} Eisenbud D., Schreyer F.-O. \emph{Resultants and Chow forms via exterior syzygies.} J. Amer. Math. Soc. 16 (2003), no. 3, 537-579.

\bibitem{Iena2010} Iena O., \emph{Vector bundles on elliptic curves and factors of automorphy}, \href{http://arxiv.org/abs/1009.3230}{arxiv.org/abs/1009.3230}.

\bibitem{O2002} Odesskii A.V., \emph{Elliptic algebras}, Russian Math. Surveys 57:6 1127-1162.

\bibitem{Perkinson} Perkinson D. \emph{Curves in Grassmannians}, Trans. of the AMS, Vol. 347, No. 9 (Sep., 1995), pp. 3179-3246


\end{thebibliography}
\end{document}